\documentclass[a4paper,11pt]{article}%
\usepackage{amsmath}
\usepackage{amsfonts}
\usepackage{amssymb}
\usepackage{graphicx}
\usepackage{hyperref}%
\providecommand{\U}[1]{\protect\rule{.1in}{.1in}}
\newtheorem{theorem}{Theorem}[section]

\newtheorem{lemma}[theorem]{Lemma}

\newtheorem{proposition}[theorem]{Proposition}
\newtheorem{remark}[theorem]{Remark}

\numberwithin{equation}{section}
\newenvironment{proof}[1][Proof]{\noindent\textbf{#1.} }{\ \rule{0.5em}{0.5em}}
\ifx\pdfoutput\relax\let\pdfoutput=\undefined\fi
\newcount\msipdfoutput
\ifx\pdfoutput\undefined\else
\ifcase\pdfoutput\else
\ifx\paperwidth\undefined\else
\ifdim\paperheight=0pt\relax\else\pdfpageheight\paperheight\fi
\ifdim\paperwidth=0pt\relax\else\pdfpagewidth\paperwidth\fi
\fi\fi\fi
\begin{document}

\title{On a global gradient estimate in $p$-Laplacian problems}
\author{Grey Ercole\\{\small Universidade Federal de Minas Gerais,}\\{\small Belo Horizonte, MG, 30.123-970, Brazil}\\{\small grey@mat.ufmg.br} }
\maketitle

\begin{abstract}
\noindent We make explicit the $p$-dependence of $C$ in the gradient estimate
$\left\Vert \nabla u\right\Vert _{\infty}^{p-1}\leq C\left\Vert f\right\Vert
_{N,1}$ by Cianchi and Maz'ya (2011). In such inequality, the constant $C$ is
uniform with respect to $f\in L^{N,1}(\Omega),$ and $u$ is the weak solution
to the Poisson equation $-\operatorname{div}(\left\vert \nabla u\right\vert
^{p-2}\nabla u)=f$ \ in a bounded domain $\Omega\subset\mathbb{R}^{N},$
$N\geq3,$ coupled with either Neumann or Dirichlet homogeneous boundary
conditions. The case $N=2$ with $f\in L^{q}(\Omega),$ for some $q>2,$ is also
considered . $\bigskip$


\noindent{\small \textbf{K}\textbf{eywords:} Dirichlet problem, gradient
estimate, Neumann problem, p-Laplacian}.

\end{abstract}

\section{Introduction}

In \cite{CM}, Cianchi and Maz'ya considered the boundary value problems
\begin{equation}
\left\{
\begin{array}
[c]{lll}%
-\operatorname{div}\left(  a(\left\vert \nabla u\right\vert )\nabla u\right)
=f(x) & \text{in} & \Omega\\
\dfrac{\partial u}{\partial\nu}=0 & \text{on} & \partial\Omega,
\end{array}
\right.  \label{na}%
\end{equation}
and%
\begin{equation}
\left\{
\begin{array}
[c]{lll}%
-\operatorname{div}\left(  a(\left\vert \nabla u\right\vert )\nabla u\right)
=f(x) & \text{in} & \Omega\\
u=0 & \text{on} & \partial\Omega,
\end{array}
\right.  \label{da}%
\end{equation}
where $\Omega$ is a bounded domain of $\mathbb{R}^{N},$ $N\geq3,$ $f\in
L^{N,1}(\Omega),$ and $a:(0,\infty)\rightarrow(0,\infty)$ is a function of
class $C^{1}$ satisfying%
\begin{equation}
-1<i_{a}:=\inf_{t>0}\frac{ta^{\prime}(t)}{a(t)}\leq s_{a}:=\sup_{t>0}%
\frac{ta^{\prime}(t)}{a(t)}<\infty. \label{iasa}%
\end{equation}

Assuming that $\partial\Omega\in W^{2}L^{N-1,1}$ and $\int_{\Omega
}f(x)\mathrm{d}x=0,$ Cianchi and Maz'ya proved the estimate%
\begin{equation}
\left\Vert \nabla u\right\Vert _{\infty}\leq Cb^{-1}(\left\Vert f\right\Vert
_{N,1}) \label{cm}%
\end{equation}
for a weak solution $u$ to the Neumann problem (\ref{na}), where
$b(t):=a(t)t,$ $t>0,$ and $C=C(\Omega,i_{a},s_{a})$ is an abstract constant
that depends on $\Omega,$ $i_{a}$ and $s_{a}.$ This result is stated in
Theorem 1.1 of \cite{CM}.

As remarked in \cite{CM}, the assumption $\partial\Omega\in W^{2}L^{N-1,1}$
means that the boundary of $\Omega$ is locally the subgraph of a function of
$N-1$ variables whose derivatives up to second order are in the Lorentz space
$L^{N-1,1}.$ Moreover, this is the weakest possible integrability assumption
on the second-order derivatives under which $\partial\Omega\in C^{1,0}$ (see
\cite{CP}).

By a small change in the proof of Theorem 1.1, Cianchi and Maz'ya obtained the
same estimate (\ref{cm}) for the weak solution of the Dirichlet problem
(\ref{da}), with $\partial\Omega\in W^{2}L^{N-1,1}$ and $f\in L^{N,1}(\Omega)$
(see Theorem 1.3 of \cite{CM}). Further, they showed that (\ref{cm}) also
holds in both problems if the hypothesis $\partial\Omega\in W^{2}L^{N-1,1}$ is
replaced with the hypothesis that $\Omega$ is convex (see Theorems 1.2 and 1.4
of \cite{CM}).

Under additional regularity conditions on $a,$ $f$ and $\partial\Omega,$ the
crucial arguments for achieving (\ref{cm}) are developed by Cianchi and Maz'ya
in a first step of Section 4 of \cite{CM}, supported by estimates established
in Section 2 of that paper. The extra regularity assumptions are then removed
in three more steps of Section 4 by approximation arguments.

When restricted to the $p$-Laplacian operator, that is, $a(t)=t^{p-2},$ $p>1,$
the estimate (\ref{cm}) is equivalent to
\begin{equation}
\left\Vert \nabla u\right\Vert _{\infty}^{p-1}\leq C\left\Vert f\right\Vert
_{N,1} \label{CM1}%
\end{equation}
with the constant $C$ depending on $N,$ $\Omega$ and $p.$ In this case,
$i_{a}=s_{a}=p-2.$

Our main goal in the present paper is to exhibit an explicit expression (not
necessarily optimal) for dependence of the constant $C$ with respect to $p,$
for the $p$-Laplacian. \ The well-known regularization $a_{\epsilon
}(t):=(t^{2}+\epsilon)^{(p-2)/2}$ of $a(t)=t^{p-2}$ provides estimates that
allow us to track the proofs given by Cianchi and Maz'ya in \cite{CM} and then
achieve our main results, stated as follows.

\begin{theorem}
\label{main}Let $\Omega$ be a bounded domain of $\mathbb{R}^{N},$ $N\geq3,$
such that $\partial\Omega\in W^{2}L^{\theta,1},$ for some $\theta>N-1.$ Assume
that $f\in L^{N,1}(\Omega)$ fulfills the compatibility condition $\int
_{\Omega}f(x)\mathrm{d}x=0.$ Let $u_{p}\in W^{1,p}(\Omega)$ be a weak solution
to the Neumann problem
\begin{equation}
\left\{
\begin{array}
[c]{lll}%
-\operatorname{div}(\left\vert \nabla u\right\vert ^{p-2}\nabla u)=f(x) &
\text{in} & \Omega\\
\dfrac{\partial u}{\partial\nu}=0 & \text{on} & \partial\Omega.
\end{array}
\right.  \label{nf}%
\end{equation}
Then there exists a constant $C$ depending at most on $N$ and $\Omega$ such
that%
\begin{equation}
\left\Vert \nabla u_{p}\right\Vert _{\infty}^{p-1}\leq C\left\{
\begin{array}
[c]{lll}%
2^{\frac{p}{p-1}}(p-1)^{-\frac{\theta N}{\theta-(N-1)}}\left\Vert f\right\Vert
_{N,1} & \text{if} & 1<p<2\\
&  & \\
p^{\frac{5}{2}+\frac{\theta N}{\theta-(N-1)}}\left\Vert f\right\Vert _{N,1} &
\text{if} & p\geq2.
\end{array}
\right.  \label{Cdp}%
\end{equation}
Moreover, if the assumption $\partial\Omega\in W^{2}L^{\theta,1}$ is replaced
with the assumption that is $\Omega$ convex, then
\begin{equation}
\left\Vert \nabla u_{p}\right\Vert _{\infty}^{p-1}\leq C\left\{
\begin{array}
[c]{lll}%
2^{\frac{p}{p-1}}\left\Vert f\right\Vert _{N,1} & \text{if} & 1<p<2\\
p^{\frac{5}{2}}\left\Vert f\right\Vert _{N,1} & \text{if} & p\geq2
\end{array}
\right.  \label{CCdp}%
\end{equation}
where $C$ is a constant depending at most on $N$ and $\Omega.$
\end{theorem}

\begin{theorem}
\label{main2}Let $\Omega$ be a bounded domain of $\mathbb{R}^{N},$ $N\geq3,$
such that $\partial\Omega\in W^{2}L^{\theta,1}$ for some $\theta>N-1,$ and
assume that $f\in L^{N,1}(\Omega).$ Let $u_{p}\in W_{0}^{1,p}(\Omega)$ be a
weak solution of the Dirichlet problem
\begin{equation}
\left\{
\begin{array}
[c]{lll}%
-\operatorname{div}(\left\vert \nabla u\right\vert ^{p-2}\nabla u)=f(x) &
\mathrm{in} & \Omega\\
u=0 & \mathrm{on} & \partial\Omega.
\end{array}
\right.  \label{df}%
\end{equation}
Then the estimate (\ref{Cdp}) holds for a constant $C,$ depending at most on
$N$ and $\Omega.$ Moreover, if the assumption $\partial\Omega\in
W^{2}L^{\theta,1}$ is replaced with the assumption that is $\Omega$ convex,
then (\ref{CCdp}) holds for constant $C$ depending at most on $N$ and
$\Omega.$
\end{theorem}

Our approach to determine how $C$ depends on $p$ in the estimate (\ref{CM1})
relies on the arguments by Cianchi and Maz'ya in \cite{CM}. Essentially, we
track their proofs and identify, for the regularization $a_{\epsilon
}(t):=(t^{2}+\epsilon)^{(p-2)/2},$ the dependence on $p$ in each step.
Following this plan we also were able to deduce estimates depending explicitly
on $i_{a}$ and $s_{a}$ for a more general function $a$ satisfying
(\ref{iasa}), not necessarily related to the $p$-Laplacian. We will commented
about this further below in this introduction.

In the case where $\Omega\ $is not convex we face a difficulty when the
boundary regularity comes into play. The issue has to do with the estimation
of $s_{p}^{-1}$ where $s_{p}$ is a positive parameter that appears from an
interaction between certain constants depending on $p$ and a nonnegative
function $k$ associated with the curvature of $\partial\Omega.$ We overcome
this difficulty by assuming that $\partial\Omega\in W^{2}L^{\theta,1}$ for
some $\theta>N-1.$ This assumption, which is lightly stronger than
$\partial\Omega\in W^{2}L^{N-1,1},$ guarantees that $k\in L^{\theta
,1}(\partial\Omega)$ and then the embedding $L^{\theta,1}\hookrightarrow
L^{N-1,1}$ turns out to provide estimates to $s_{p}^{-1}$ that are explicit
with respect to $p.$

The approach used by Cianchi and Maz'ya in \cite{CM} does not work directly
when $N=2.$ The reason is that the arguments used to estimate certain
quantities involving\ $f^{\ast}$ (the decreasing rearrangement of $f$ ) in
terms of $\left\Vert f\right\Vert _{N,1}$ call for $N>2$ (see Lemmas 3.5 and
3.6 of that paper). However, as remarked by Cianchi and Maz'ya in \cite{CM15}
a version of their results in \cite{CM} also holds when $N=2$ if the
assumption $f\in L^{2,1}(\Omega)$ is replaced with the slightly stronger
assumption $f\in L^{q}(\Omega)$ for some some $q>2.$ It is worth mentioning
that in \cite{Mz2} Maz'ya derived the estimate $\left\Vert \nabla u\right\Vert
_{\infty}\leq C\left\Vert f\right\Vert _{q}$ for the weak solution $u$ to the
Neumann problem for the Laplacian (i.e. $p=2$) under the assumptions: $\Omega$
convex and $f\in L^{q}(\Omega)$ for some $q>N\geq2$. Inspired by these facts
we also consider the case $N=2$ and obtain the following result where
$\left\Vert \cdot\right\Vert _{q}$ denotes the standard norm of $L^{q}%
(\Omega)$.

\begin{theorem}
\label{N=2}Let $\Omega$ be a bounded domain of $\mathbb{R}^{2}$ such that
$\partial\Omega\in W^{2}L^{\theta,1},$ for some $\theta>1.$ Let $f\in
L^{q}(\Omega),$ for some $q>2,$ and let $u_{p}$ be either a weak solution of
the Neumann problem (\ref{nf}), under the compatibility condition
$\int_{\Omega}f(x)\mathrm{d}x=0,$ or a solution of the Dirichlet problem
(\ref{df}). Then,
\begin{equation}
\left\Vert \nabla u_{p}\right\Vert _{\infty}^{p-1}\leq C\left\{
\begin{array}
[c]{lll}%
2^{\frac{p}{p-1}}(p-1)^{-\frac{2\theta}{\theta-1}}\left\Vert f\right\Vert _{q}
& \text{if} & 1<p<2\\
&  & \\
p^{\frac{5}{2}+\frac{2\theta}{\theta-1}}\left\Vert f\right\Vert _{q} &
\text{if} & p\geq2,
\end{array}
\right.  \label{C2a}%
\end{equation}
for some constant $C$ depending at most on $q$ and $\Omega.$ Moreover, if the
assumption $\partial\Omega\in W^{2}L^{\theta,1}$ is replaced with the
assumption that is $\Omega$ convex, then
\begin{equation}
\left\Vert \nabla u_{p}\right\Vert _{\infty}^{p-1}\leq C\left\{
\begin{array}
[c]{lll}%
2^{\frac{p}{p-1}}\left\Vert f\right\Vert _{q} & \text{if} & 1<p<2\\
p^{\frac{5}{2}}\left\Vert f\right\Vert _{q} & \text{if} & p\geq2,
\end{array}
\right.  \label{C2b}%
\end{equation}
for some constant $C$ depending at most on $q$ and $\Omega.$
\end{theorem}

We believe that our results above can be useful in problems involving the
behavior of solutions to (\ref{nf}) and (\ref{df}) as $p\rightarrow1^{+}$ or
$p\rightarrow+\infty.$ In \cite{AEJ} they are crucial to determine the
limiting behavior, as $p\rightarrow+\infty,$ of the solution $u_{p}$ to a
family of $p$-Laplacian problems involving gradient and exponential terms.

Our proofs exploit the structural properties
\begin{equation}
-1<i_{a_{\epsilon}}:=\inf_{t>0}\frac{ta_{\epsilon}^{\prime}(t)}{a_{\epsilon
}(t)}\leq s_{a_{\epsilon}}:=\sup_{t>0}\frac{ta_{\epsilon}^{\prime}%
(t)}{a_{\epsilon}(t)}<\infty\label{iasae}%
\end{equation}
of the regularization $a_{\epsilon}$ instead of its particular form. This
strategy allows us generalize our estimates for operators given by a function
$a$ (not necessarily the $p$-Laplacian). In fact, we obtain the global
estimate
\begin{equation}
b(\left\Vert \nabla u\right\Vert _{\infty})\leq C\Lambda(i_{a},s_{a}%
)\left\Vert f\right\Vert _{N,1} \label{Gama}%
\end{equation}
for a solution $u$ to either (\ref{nf}) or (\ref{df}). Here, $C$ depends at
most on $N$ and $\Omega,$ and $\Lambda(i_{a},s_{a})$ is given explicitly in
terms of $i_{a}$ and $s_{a}.$

This paper is organized as follows. In Section \ref{sec1} we derive some
inequalities that are the counterparts of those developed in Section 2 of
\cite{CM}. In Section \ref{sec2} we reproduce the first step of Section 4 of
\cite{CM} to achieve (\ref{CM1}) under regularity assumptions. In the
sequence, by following the remaining steps in Section 4 of \cite{CM}, we
present the proofs of Theorems \ref{main} and \ref{main2}, in Subsection
\ref{sec3.1}, and of Theorem \ref{N=2} in Subsection \ref{sec3.2}. In Section
\ref{aeps1} we indicate how to arrive at (\ref{Gama}).

\section{Preliminaries\label{sec1}}

For $\epsilon>0$ let us define the function $a_{\epsilon}:[0,\infty
)\rightarrow(0,\infty)$ as%
\begin{equation}
a_{\epsilon}(t):=(t^{2}+\epsilon)^{\frac{p-2}{2}} \label{a}%
\end{equation}
where $p>1.$ Note that
\begin{equation}
a_{\epsilon}\in C^{\infty}([0,\infty))\text{ \ and \ }a_{\epsilon}(0)>0.
\label{cinf}%
\end{equation}

As%
\[
\frac{ta_{\epsilon}^{\prime}(t)}{a_{\epsilon}(t)}=(p-2)\frac{t^{2}}%
{t^{2}+\epsilon}\quad\forall\,t>0
\]
one can readily verify that $a_{\epsilon}$ satisfies (\ref{iasae}). In fact,
one has
\begin{equation}
i_{a_{\epsilon}}:=\inf_{t>0}\frac{ta_{\epsilon}^{\prime}(t)}{a_{\epsilon}%
(t)}=\min\left\{  p-2,0\right\}  >-1 \label{iae}%
\end{equation}
and%
\begin{equation}
s_{a_{\epsilon}}:=\sup_{t>0}\frac{ta_{\epsilon}^{\prime}(t)}{a_{\epsilon}%
(t)}=\max\left\{  p-2,0\right\}  <\infty. \label{sae}%
\end{equation}
Moreover, as $i_{a}=i_{b}=p-2,$ one has
\begin{equation}
\min\left\{  i_{a},0\right\}  =i_{a_{\epsilon}}\leq s_{a_{\epsilon}}%
=\max\left\{  s_{a},0\right\}  . \label{iaeps}%
\end{equation}

Let us define the strictly increasing functions $b_{\epsilon}:[0,\infty
)\rightarrow\lbrack0,\infty)$ and $B_{\epsilon}:[0,\infty)\rightarrow
\lbrack0,\infty),$ respectively as%
\begin{equation}
b_{\epsilon}(t)=a_{\epsilon}(t)t \label{b}%
\end{equation}
and%
\[
B_{\epsilon}(t)=\int_{0}^{t}b_{\epsilon}(\tau)\mathrm{d}\tau.
\]
The monotonicity of $b_{\epsilon}$ follows from (\ref{iae}) as
\[
b_{\epsilon}^{\prime}(t)=a_{\epsilon}(t)+ta_{\epsilon}^{\prime}(t)=a_{\epsilon
}(t)\left(  1+\frac{ta_{\epsilon}^{\prime}(t)}{a_{\epsilon}(t)}\right)  \geq
a_{\epsilon}(t)\left(  1+i_{a_{\epsilon}}\right)  >0.
\]

It is simple to check that
\begin{equation}
\lim_{\epsilon\rightarrow0^{+}}b_{\epsilon}(t)=t^{p-1}\text{\ uniformly in
}[0,M]\text{ for every }M>0 \label{bconv}%
\end{equation}
and%
\begin{equation}
\lim_{\epsilon\rightarrow0^{+}}B_{\epsilon}(t)=\frac{t^{p}}{p}%
\text{\ uniformly in }[0,M]\text{ for every }M>0. \label{Bconv}%
\end{equation}
Moreover, it follows from (\ref{bconv}) that%

\begin{equation}
\lim_{\epsilon\rightarrow0^{+}}a_{\epsilon}(\left\vert x\right\vert
)x=\left\vert x\right\vert ^{p-2}x\text{ \ uniformly in }\left\{
x\in\mathbb{R}^{N}:\left\vert x\right\vert \leq M\right\}  \text{ for every
}M>0. \label{aconv}%
\end{equation}

We remark from (\ref{cinf})-(\ref{aconv}) that the conclusions of Lemma 3.3 of
\cite{CM} hold for $a_{\epsilon},$ $b_{\epsilon}$ and $B_{\epsilon}.$

\begin{lemma}
\label{Lem1}Let $\psi_{\epsilon}:[0,\infty)\rightarrow\lbrack0,\infty)$ be the
function defined as%
\begin{equation}
\psi_{\epsilon}(s):=sb_{\epsilon}^{-1}(s). \label{psi}%
\end{equation}
If $C\geq1,$ then%
\begin{equation}
C\psi_{\epsilon}(s)\leq\psi_{\epsilon}(Cs)\quad\forall\text{ }s\geq0.
\label{psitil}%
\end{equation}

\end{lemma}

\begin{proof}
As $b^{-1}:[0,\infty)\rightarrow\lbrack0,\infty)$ is increasing and $C\geq1$,
one has
\[
C\psi_{\epsilon}(s)=Csb_{\epsilon}^{-1}(s)\leq Csb_{\epsilon}^{-1}%
(Cs)=\psi_{\epsilon}(Cs)\quad\forall\text{ }s\geq0.
\]

\end{proof}

\begin{lemma}
\label{h0}Let $h:(0,\infty)\rightarrow(0,\infty)$ be a function of class
$C^{1}$ such that%
\begin{equation}
\alpha\leq\frac{th^{\prime}(t)}{h(t)}\leq\beta\quad\forall\,t>0. \label{h}%
\end{equation}
Then,
\begin{equation}
\min\left\{  c^{\alpha},c^{\beta}\right\}  \leq\frac{h(\rho t)}{h(t)}\leq
\max\left\{  c^{\alpha},c^{\beta}\right\}  \quad\forall\,c,t>0. \label{h1}%
\end{equation}

\end{lemma}

\begin{proof}
It follows from (\ref{h}) that%
\begin{equation}
\frac{d}{dt}\log(t^{\alpha})\leq\frac{d}{dt}\log(h(t))\leq\frac{d}{dt}%
\log(t^{\alpha}). \label{h4}%
\end{equation}
If $0<c\leq1,$ then we integrate (\ref{h4}) over the interval $[ct,t]$ to
obtain the inequalities%
\[
\log(c^{-\alpha})\leq\log(h(t)/h(ct))\leq\log(c^{-\beta})
\]
which leads to (\ref{h1}) after exponentiation. Analogously, the integration
of (\ref{h4}) over $[t,ct]$ followed by exponentiation yields (\ref{h1}) if
$c>1.$
\end{proof}

We note that%
\[
\frac{tb_{\epsilon}^{\prime}(t)}{b_{\epsilon}(t)}=t\frac{a_{\epsilon
}(t)+ta_{\epsilon}^{\prime}(t)}{ta_{\epsilon}(t)}=1+\frac{ta_{\epsilon
}^{\prime}(t)}{a_{\epsilon}(t)}%
\]
so that%
\begin{equation}
\min\left\{  1,p-1\right\}  =1+i_{a_{\epsilon}}\leq\frac{tb_{\epsilon}%
^{\prime}(t)}{b_{\epsilon}(t)}\leq1+s_{a_{\epsilon}}=\max\left\{
1,p-1\right\}  \quad\forall\,t>0. \label{h2}%
\end{equation}

Thus, it follows from Lemma \ref{h0} that%
\begin{equation}
\min\left\{  c,c^{p-1}\right\}  \leq\frac{b_{\epsilon}(ct)}{b_{\epsilon}%
(t)}\leq\max\left\{  c,c^{p-1}\right\}  \quad\forall\,c,t>0. \label{h5}%
\end{equation}

In the sequel
\[
m(c,p):=\min\left\{  2c,2c^{p-1},pc,pc^{p-1}\right\}
\]
and
\[
M(c,p):=\max\left\{  2c,2c^{p-1},pc,pc^{p-1}\right\}  .
\]

\begin{lemma}
For each $c>0$ one has%
\begin{equation}
m(c,p)\leq\frac{tb_{\epsilon}(ct)}{B_{\epsilon}(t)}\leq M(c,p)\quad
\forall\,t>0. \label{Bbcp}%
\end{equation}

\end{lemma}

\begin{proof}
It also follows from (\ref{h2}) that%
\[
\min\left\{  1,p-1\right\}  B_{\epsilon}(t)\leq\int_{0}^{t}sb_{\epsilon
}^{\prime}(s)\mathrm{d}s\leq B_{\epsilon}(t)\max\left\{  1,p-1\right\}
\]
Hence, as%
\[
\int_{0}^{t}sb_{\epsilon}^{\prime}(s)\mathrm{d}s=tb_{\epsilon}(t)-B_{\epsilon
}(t)
\]
we obtain the bounds%
\begin{equation}
\min\left\{  2,p\right\}  \leq\frac{tb_{\epsilon}(t)}{B_{\epsilon}(t)}\leq
\max\left\{  2,p\right\}  . \label{h6}%
\end{equation}

Noticing that
\[
\frac{tb_{\epsilon}(ct)}{B_{\epsilon}(t)}=\frac{tb_{\epsilon}(t)}{B_{\epsilon
}(t)}\frac{b_{\epsilon}(ct)}{b_{\epsilon}(t)}%
\]
we gather (\ref{h5}) and (\ref{h6}) to produce the estimates%
\[
\min\left\{  c,c^{p-1}\right\}  \min\left\{  2,p\right\}  \leq\frac
{tb_{\epsilon}(ct)}{B_{\epsilon}(t)}\leq\max\left\{  2,p\right\}  \max\left\{
c,c^{p-1}\right\}
\]
which lead to (\ref{Bbcp}).
\end{proof}

Let $\widehat{B_{\epsilon}}:[0,\infty)\rightarrow\lbrack0,\infty)$ be function
defined as%
\[
\widehat{B_{\epsilon}}(0)=0\text{ \ and \ }\widehat{B_{\epsilon}}%
(t):=\frac{B_{\epsilon}(t)}{t}\quad\forall\,t>0,
\]
and let $F_{\epsilon}:[0,\infty)\rightarrow\lbrack0,\infty)$ be the function
given by
\begin{equation}
F_{\epsilon}(t):=\int_{0}^{t}b_{\epsilon}(s)^{2}\mathrm{d}s\quad\forall
\,t\geq0. \label{F}%
\end{equation}

\begin{proposition}
\label{Prop2}One has%
\begin{equation}
\min\left\{  2,p\right\}  B_{\epsilon}(t)\leq tb_{\epsilon}(t)\leq\max\left\{
2,p\right\}  B_{\epsilon}(t)\quad\forall\,t\geq0, \label{Bb}%
\end{equation}%
\begin{equation}
\widehat{B_{\epsilon}}^{-1}(s)\leq C_{p}b_{\epsilon}^{-1}(s)\quad\forall\,s>0,
\label{B^}%
\end{equation}
and%
\begin{equation}
F_{\epsilon}(t)\leq tb_{\epsilon}(t)^{2}\leq K_{p}F_{\epsilon}(t)\quad
\forall\,t\geq0, \label{Fb}%
\end{equation}
where%
\begin{equation}
C_{p}:=\left\{
\begin{array}
[c]{lll}%
2^{\frac{1}{p-1}} & \mathrm{if} & 1<p<2\\
p & \mathrm{if} & p\geq2,
\end{array}
\right.  \label{Cp}%
\end{equation}
and%
\begin{equation}
K_{p}:=\left\{
\begin{array}
[c]{lll}%
3 & \mathrm{if} & 1<p<2\\
2p-1 & \mathrm{if} & p\geq2.
\end{array}
\right.  \label{C'p}%
\end{equation}

\end{proposition}

\begin{proof}
The estimates in (\ref{Bb}) come directly from (\ref{Bbcp}) with $c=1.$
Inequality (\ref{B^}) is equivalent to%
\begin{equation}
\frac{tb_{\epsilon}(t/C_{p})}{B_{\epsilon}(t)}\leq1\quad\forall\,t>0,
\label{B^1}%
\end{equation}
($t=C_{p}b_{\epsilon}^{-1}(s)$). The second inequality in (\ref{Bbcp}), with
$c=C_{p}^{-1},$ yields%
\begin{equation}
\frac{tb_{\epsilon}(t/C_{p})}{B_{\epsilon}(t)}\leq M(C_{p}^{-1},p).
\label{B^2}%
\end{equation}

If $1<p<2$ we have%
\[
M(C_{p}^{-1},p)\leq\max\left\{  2C_{p}^{-1},2C_{p}^{-(p-1)}\right\}
=\max\left\{  2^{\frac{p-2}{p-1}},1\right\}  =1
\]
and if $p\geq2$ we have%
\[
M(C_{p}^{-1},p)\leq\max\left\{  pC_{p}^{-1},pC_{p}^{-(p-1)}\right\}
=\max\left\{  1,p^{2-p}\right\}  =1.
\]

In both cases we obtain (\ref{B^1}) from (\ref{B^2}), so that (\ref{B^}) is proved.

The first inequality in (\ref{Fb}) follows from the fact that $b_{\epsilon
}^{2}$ is increasing. In order to prove the second one we note from (\ref{h2})
that%
\[
tb_{\epsilon}^{\prime}(t)\leq\max\left\{  1,p-1\right\}  b_{\epsilon}%
(t)\quad\forall\,t>0.
\]
Hence, integration by parts yields%
\begin{align*}
F_{\epsilon}(t)  &  =tb_{\epsilon}(t)^{2}-\int_{0}^{t}2sb_{\epsilon
}(s)b_{\epsilon}^{\prime}(s)\mathrm{d}s\\
&  \geq tb_{\epsilon}(t)^{2}-2\max\left\{  1,p-1\right\}  \int_{0}%
^{t}b_{\epsilon}^{2}(s)\mathrm{d}s\\
&  =tb_{\epsilon}(t)^{2}-2\max\left\{  1,p-1\right\}  F_{\epsilon}%
(t),\quad\forall\,t>0.
\end{align*}

Consequently,%
\[
\max\left\{  3,2p-1\right\}  F_{\epsilon}(t)\geq tb_{\epsilon}(t)^{2}.
\]

\end{proof}

\begin{remark}
\label{aeps}The proofs presented in this section enable us to write some
constants and estimates in terms of $i_{a_{\epsilon}}$ and $s_{a_{\epsilon}}$
for an arbitrary regularization $a_{\epsilon}$ of $a$ (not necessarily related
to the $p$-Laplacian) satisfying the structural conditions (\ref{iasa}). Thus,
we have that:

\begin{itemize}
\item $\min\left\{  c^{1+i_{a_{\epsilon}}},c^{1+s_{a_{\epsilon}}}\right\}
(2+i_{a_{\epsilon}})\leq\frac{tb_{\epsilon}(ct)}{B_{\epsilon}(t)}%
\leq(2+s_{a_{\epsilon}})\max\left\{  c^{1+i_{a_{\epsilon}}}%
,c^{1+s_{a_{\epsilon}}}\right\}  $ in (\ref{Bbcp}),

\item $C_{\epsilon}:=(2+s_{a_{\epsilon}})^{1/(1+i_{a_{\epsilon}})}$ is the
constant equivalent to $C_{p}$ in (\ref{Cp}),

\item $K_{\epsilon}:=3+2s_{a_{\epsilon}}$ is the constant equivalent to
$K_{p}$ in (\ref{C'p}).
\end{itemize}
\end{remark}

\section{Proofs\label{sec2}}

Let $\left(  \mathcal{R},\mathfrak{m}\right)  $ be a finite, positive measure
space and let $l>1.$ We recall that the Lorentz space $L^{l,1}(\mathcal{R})$
consists of all measurable functions $v:\mathcal{R}\rightarrow\mathbb{R}$ such
that%
\[
\int_{0}^{\mathfrak{m}(\mathcal{R})}\tau^{-1/l^{\prime}}\left\vert v^{\ast
}(\tau)\right\vert \mathrm{d}\tau<\infty.
\]
Here, $l^{\prime}:=\frac{l}{l-1}$ and $v^{\ast}:[0,\infty)\rightarrow
\lbrack0,\infty]$ stands for the decreasing rearrangement of $v$, which is
defined as%
\[
v^{\ast}(s):=\left\{
\begin{array}
[c]{lll}%
\sup\left\{  t\geq0:\mu_{v}(t)>s\right\}  & \text{if} & 0\leq s\leq
\mathfrak{m}(\mathcal{R})\\
0 & \text{if} & s>\mathfrak{m}(\mathcal{R}),
\end{array}
\right.
\]
where
\[
\mu_{v}(t):=\mathfrak{m}\left(  \left\{  x\in\mathcal{R}:v(x)>t\right\}
\right)  ,\quad t\geq0,
\]
is the distribution function of $v.$

As it is well known, $L^{l,1}(\mathcal{R})$ is a Banach space endowed with the
norm
\[
\left\Vert v\right\Vert _{l,1}:=\int_{0}^{\mathfrak{m}(\mathcal{R})}\left\vert
v^{\ast\ast}(\tau)\right\vert \tau^{-1/l^{\prime}}\mathrm{d}\tau
\]
where $v^{\ast\ast}:(0,\infty)\rightarrow\lbrack0,\infty)$ is defined as%
\[
v^{\ast\ast}(s):=\frac{1}{s}\int_{0}^{s}v^{\ast}(r)\mathrm{d}r,\quad s>0.
\]

Let $\Omega$ be a bounded domain of $\mathbb{R}^{N},$ $N\geq2,$ and let $f\in
L^{N,1}(\Omega).$ For each $\epsilon>0$, let us consider the boundary value
problems (\ref{na}) and (\ref{da}) with the particular function $a=a_{\epsilon
}$ defined in (\ref{a}):
\begin{equation}
\left\{
\begin{array}
[c]{lll}%
-\operatorname{div}\left(  \left\vert (\nabla u\right\vert ^{2}+\epsilon
)^{\frac{p-2}{2}}\nabla u\right)  =f(x) & \mathrm{in} & \Omega\\
\dfrac{\partial u}{\partial\nu}=0 & \mathrm{on} & \partial\Omega,
\end{array}
\right.  \label{nae}%
\end{equation}
and%
\begin{equation}
\left\{
\begin{array}
[c]{lll}%
-\operatorname{div}\left(  \left\vert (\nabla u\right\vert ^{2}+\epsilon
)^{\frac{p-2}{2}}\nabla u\right)  =f(x) & \mathrm{in} & \Omega\\
u=0 & \mathrm{on} & \partial\Omega.
\end{array}
\right.  \label{dae}%
\end{equation}

If $\partial\Omega$ is at least Lipschitz, then existence and uniqueness of a
weak solution to (\ref{nae}) in $W_{\bot}^{1,p}(\Omega):=\left\{  u\in
W^{1,p}(\Omega):\int_{\Omega}u(x)\mathrm{d}x=0\right\}  $ and to (\ref{dae})
in $W_{0}^{1,p}(\Omega)$ are well known facts. As for (\ref{nae}) it is
assumed that $f$ fulfills the compatibility condition
\begin{equation}
\int_{\Omega}f(x)\mathrm{d}x=0. \label{f0}%
\end{equation}

The next statement is the reproduction of Theorem 2.14 of \cite{CM} taking
into account the definition of $\psi_{\epsilon}$ in (\ref{psi}) and the
estimate (\ref{B^}).

\begin{lemma}
\label{2.14} Let $u_{\epsilon}$ denote either the weak solution to the Neumann
problem (\ref{nae}) in $W_{\bot}^{1,p}(\Omega)$ or the solution to the
Dirichlet problem (\ref{dae}) in $W_{0}^{1,p}(\Omega)$. Then%
\begin{equation}
\int_{\Omega}B_{\epsilon}(\left\vert \nabla u_{\epsilon}\right\vert
)\mathrm{d}x\leq C^{\prime\prime}C_{p}\psi_{\epsilon}(\left\Vert f\right\Vert
_{N,1}) \label{aux1}%
\end{equation}
where $C_{p}$ is defined by (\ref{Cp}) and $C^{\prime\prime}$ is a constant
that depends at most on $N$ and $\Omega.$
\end{lemma}

From now on, $\left\vert \Omega\right\vert $ will denote the $N$-dimensional
Lebesgue measure of $\Omega$ and $\left\vert \partial\Omega\right\vert $ will
denote the $(N-1)$-dimensional Haursdorff measure of $\partial\Omega.$

\subsection{Proofs of Theorems \ref{main} and \ref{main2}\label{sec3.1}}

In this subsection we fix $N\geq3$ and $\theta>N-1.$ Let us assume for a while
that%
\begin{equation}
\partial\Omega\in C^{\infty} \label{regb}%
\end{equation}
and%
\begin{equation}
f\in C_{c}^{\infty}(\Omega). \label{regf}%
\end{equation}
To simplify the notation we drop the subscript $\epsilon$ of $u_{\epsilon}$.

According to \cite{CM}, the assumptions (\ref{regb})-(\ref{regf}) guarantee
that $u\in C^{3}(\overline{\Omega})$ and in addition
\begin{equation}
\Delta u\frac{\partial u}{\partial\nu}-\sum_{i,j}u_{x_{i}x_{j}}u_{x_{i}}%
\nu_{j}=-\mathcal{B}\left(  \nabla_{T}u,\nabla_{T}u\right)  \text{ \ on
}\partial\Omega\label{dtrace1}%
\end{equation}
if $u$ is the weak solution to (\ref{nae}) and
\begin{equation}
\Delta u\frac{\partial u}{\partial\nu}-\sum_{i,j}u_{x_{i}x_{j}}u_{x_{i}}%
\nu_{j}=-\operatorname{tr}\mathcal{B}\left(  \frac{\partial u}{\partial\nu
}\right)  ^{2}\text{ \ on }\partial\Omega\label{dtrace}%
\end{equation}
if $u$ is the weak solution to (\ref{dae}). Here, $\nu_{j}$ denotes the $j$th
component of the normal vector $\nu$ to $\partial\left\{  \left\vert \nabla
u\right\vert >t\right\}  ,$ $\mathcal{B}$ denotes the second fundamental form
of $\partial\Omega,$ $\operatorname{tr}\mathcal{B}$ denotes the trace of
$\mathcal{B}.$ In (\ref{dtrace1}) the symbol $\nabla_{T}$ stands for the
gradient operator on $\partial\Omega.$

As argued in \cite{CM},
\begin{equation}
\mathcal{B}\left(  \nabla_{T}u,\nabla_{T}u\right)  \leq k(x)\left\vert
\nabla_{T}u\right\vert ^{2}\text{ \ on }\partial\Omega\label{tr1}%
\end{equation}
if $u$ is the weak solution to (\ref{nae}) and
\begin{equation}
\operatorname{tr}\mathcal{B}\left(  \frac{\partial u}{\partial\nu}\right)
^{2}\leq k(x)\left\vert \nabla u\right\vert ^{2}\text{ \ on }\partial
\Omega\label{tr}%
\end{equation}
if $u$ is the weak solution to (\ref{dae}). In both inequalities $k\in
L^{N-1,1}(\partial\Omega)$ is a nonnegative function that is pointwise
estimated, up to a multiplicative constant depending on $\partial\Omega,$ by
the second-order derivatives of the $(N-1)$-dimensional functions which
locally represent $\partial\Omega.$

In the sequel $\phi:(0,\left\vert \Omega\right\vert )\rightarrow
\lbrack0,\infty)$ is the function defined by
\[
\phi(s):=\left(  \frac{d}{ds}\int_{\left\{  \left\vert \nabla u\right\vert
>\left\vert \nabla u\right\vert ^{\ast}(s)\right\}  }f^{2}\mathrm{d}x\right)
^{2}\text{ \ for a.e. }s\in(0,\left\vert \Omega\right\vert )
\]
and $\mu$ the distribution function of $\left\vert \nabla u\right\vert .$

For each $t_{0}\in\left[  \left\vert \nabla u\right\vert ^{\ast}%
(\Omega/2),\left\Vert \nabla u\right\Vert _{\infty}\right]  $ we can proceed
as in \cite{CM}, using (\ref{tr1}) in (\ref{dtrace1}) and (\ref{tr}) in
(\ref{dtrace}), to arrive at the inequality
\begin{equation}%
\begin{array}
[c]{lll}%
2\xi_{p}F_{\epsilon}(\left\vert \nabla u\right\vert ^{\ast}(s)) & \leq &
2\xi_{p}F_{\epsilon}(t_{0})+C_{\Omega}\left\Vert \nabla u\right\Vert _{\infty
}b_{\epsilon}(\left\Vert \nabla u\right\Vert _{\infty})%
{\displaystyle\int_{s}^{\mu(t_{0})}}
r^{-1/N^{\prime}}\phi(r)\mathrm{d}r\\
&  & +\dfrac{C_{\Omega}}{\xi_{p}}\left\Vert \nabla u\right\Vert _{\infty}%
{\displaystyle\int_{s}^{\mu(t_{0})}}
r^{-2/N^{\prime}}%
{\displaystyle\int_{0}^{r}}
f^{\ast}(\rho)^{2}\mathrm{d}\rho\mathrm{d}r\\
&  & +C\left\Vert \nabla u\right\Vert _{\infty}b_{\epsilon}(\left\Vert \nabla
u\right\Vert _{\infty})^{2}%
{\displaystyle\int_{s}^{\mu(t_{0})}}
k^{\ast\ast}(c_{\Omega}r^{\frac{1}{N^{\prime}}})r^{-\frac{1}{N^{\prime}}%
}\mathrm{d}r
\end{array}
\label{aux0}%
\end{equation}
valid for every $s\in\lbrack0,\mu(t_{0})).$ Here, $b_{\epsilon}$ and
$F_{\epsilon}$ are the functions defined in (\ref{b}) and (\ref{F}), respectively,%

\begin{equation}
\xi_{p}:=\frac{1+\min\left\{  i_{a_{\epsilon}},0\right\}  }{2}=\frac
{\min\left\{  p-1,1\right\}  }{2}, \label{qsi}%
\end{equation}
$C_{\Omega}$ and $c_{\Omega}$ are positive constants depending at most on
$\Omega$ (but not on $p$).

\begin{remark}
\label{rem1}As observed in \cite{CM}, if $\Omega$ is convex then
$\mathcal{B}\left(  \nabla_{T}u,\nabla_{T}u\right)  \leq0$ and
$\operatorname{tr}\mathcal{B}\leq0$ so that the right-hand sides of
(\ref{tr1}) and (\ref{tr}) can be replaced by $0.$ Therefore, the latter term
at the right-hand side of (\ref{aux0}) can be disregarded.
\end{remark}

As $\left\vert \nabla u\right\vert ^{\ast}(0)=\left\Vert \nabla u\right\Vert
_{\infty},$ taking $s=0$ in (\ref{aux0}) and using (\ref{Fb}) we obtain%

\begin{equation}%
\begin{array}
[c]{lll}%
2\xi_{p}F_{\epsilon}(\left\Vert \nabla u\right\Vert _{\infty}) & \leq &
2\xi_{p}F_{\epsilon}(t_{0})+C_{\Omega}b_{\epsilon}(\left\Vert \nabla
u\right\Vert _{\infty})\left\Vert \nabla u\right\Vert _{\infty}%
{\displaystyle\int_{0}^{\left\vert \Omega\right\vert }}
r^{-1/N^{\prime}}\phi(r)\mathrm{d}r\\
&  & +\dfrac{C_{\Omega}}{\xi_{p}}\left\Vert \nabla u\right\Vert _{\infty}%
{\displaystyle\int_{0}^{\left\vert \Omega\right\vert }}
r^{-2/N^{\prime}}%
{\displaystyle\int_{0}^{r}}
f^{\ast}(\rho)^{2}\mathrm{d}\rho\mathrm{d}r\\
&  & +K_{p}F_{\epsilon}(\left\Vert \nabla u\right\Vert _{\infty})%
{\displaystyle\int_{0}^{\mu(t_{0})}}
k^{\ast\ast}(c_{\Omega}r^{\frac{1}{N^{\prime}}})r^{-\frac{1}{N^{\prime}}%
}\mathrm{d}r
\end{array}
\label{aux2}%
\end{equation}
where $K_{p}$ is defined in (\ref{C'p}). Note that we also have used that%
\[%
{\displaystyle\int_{s}^{\mu(t_{0})}}
r^{-1/N^{\prime}}\phi(r)\mathrm{d}r\leq%
{\displaystyle\int_{0}^{\left\vert \Omega\right\vert }}
r^{-1/N^{\prime}}\phi(r)\mathrm{d}r
\]
and%
\[%
{\displaystyle\int_{s}^{\mu(t_{0})}}
r^{-2/N^{\prime}}%
{\displaystyle\int_{0}^{r}}
f^{\ast}(\rho)^{2}\mathrm{d}\rho\mathrm{d}r\leq%
{\displaystyle\int_{0}^{\left\vert \Omega\right\vert }}
r^{-2/N^{\prime}}%
{\displaystyle\int_{0}^{r}}
f^{\ast}(\rho)^{2}\mathrm{d}\rho\mathrm{d}r.
\]

Owing to the form of the function $\phi^{2}$ one has (see \cite[Proposition
3.4]{CM})%

\begin{equation}
\int_{0}^{s}\phi^{\ast}(r)^{2}\mathrm{d}r\leq\int_{0}^{s}f^{\ast}%
(r)^{2}\mathrm{d}r\quad\forall\,s\in(0,\left\vert \Omega\right\vert ).
\label{phif}%
\end{equation}

Taking into account that $N>2,$ the inequality (\ref{phif}) implies that (see
\cite[Lemma 3.5]{CM})
\begin{equation}
\int_{0}^{\left\vert \Omega\right\vert }r^{-1/N^{\prime}}\phi(r)\mathrm{d}%
r\leq C_{N}\left\Vert f\right\Vert _{N,1} \label{fa}%
\end{equation}
for some positive constant $C_{N}$ depending only on $N.$ The assumption $N>2$
also implies that (see \cite[Lemma 3.6]{CM})
\begin{equation}
\int_{0}^{\left\vert \Omega\right\vert }r^{-2/N^{\prime}}%
{\displaystyle\int_{0}^{r}}
f^{\ast}(\rho)^{2}\mathrm{d}\rho\mathrm{d}r\leq\widetilde{C}_{N}\left\Vert
f\right\Vert _{N,1}^{2}, \label{fb}%
\end{equation}
for some positive constant $\widetilde{C}_{N}$ depending only on $N.$

\begin{remark}
\label{n2}The hypothesis $N\geq3$ is used to deduce (\ref{fa}) and (\ref{fb}).
In order to treat the case $N=2$ in the next subsection we will assume the
slightly stronger hypothesis: $f\in L^{q}(\Omega)$ for some $q>2.$
\end{remark}

Using (\ref{fa}) and (\ref{fb}) in (\ref{aux2}) we obtain
\begin{equation}%
\begin{array}
[c]{lll}%
2\xi_{p}F_{\epsilon}(\left\Vert \nabla u\right\Vert _{\infty}) & \leq &
2\xi_{p}F_{\epsilon}(t_{0})+C_{\Omega}C_{N}\left\Vert \nabla u\right\Vert
_{\infty}b_{\epsilon}(\left\Vert \nabla u\right\Vert _{\infty})\left\Vert
f\right\Vert _{N,1}\\
&  & +\dfrac{1}{\xi_{p}}C_{\Omega}\widetilde{C}_{N}\left\Vert \nabla
u\right\Vert _{\infty}\left\Vert f\right\Vert _{N,1}^{2}\\
&  & +F_{\epsilon}(\left\Vert \nabla u\right\Vert _{\infty})K_{p}%
{\displaystyle\int_{0}^{\mu(t_{0})}}
k^{\ast\ast}(c_{\Omega}r^{\frac{1}{N^{\prime}}})r^{-\frac{1}{N^{\prime}}%
}\mathrm{d}r
\end{array}
\label{aux3}%
\end{equation}
whenever $t_{0}\in\left[  \left\vert \nabla u\right\vert ^{\ast}%
(\Omega/2),\left\Vert \nabla u\right\Vert _{\infty}\right]  .$

Let $G:[0,\infty)\rightarrow\lbrack0,\infty)$ be the function%
\[
G(s):=%
{\displaystyle\int_{0}^{s}}
k^{\ast\ast}(c_{\Omega}r^{\frac{1}{N^{\prime}}})r^{-\frac{1}{N^{\prime}}%
}\mathrm{d}r=N^{\prime}(c_{\Omega})^{1-N^{\prime}}%
{\displaystyle\int_{0}^{c_{\Omega}s^{1/N^{\prime}}}}
k^{\ast\ast}(\tau)\tau^{-\frac{1}{(N-1)^{\prime}}}\mathrm{d}\tau.
\]

If $c_{\Omega}s^{1/N^{\prime}}\leq\left\vert \partial\Omega\right\vert $ then,
as $\theta>N-1$ and
\[
\left\Vert k\right\Vert _{\theta,1}=%
{\displaystyle\int_{0}^{\left\vert \partial\Omega\right\vert }}
k^{\ast\ast}(\tau)\tau^{-\frac{1}{\theta^{\prime}}}\mathrm{d}\tau,
\]
one has%
\begin{align*}
G(s)  &  =N^{\prime}(c_{\Omega})^{1-N^{\prime}}%
{\displaystyle\int_{0}^{c_{\Omega}s^{1/N^{\prime}}}}
k^{\ast\ast}(\tau)\tau^{\frac{1}{\theta^{\prime}}-\frac{1}{(N-1)^{\prime}}%
}\tau^{-\frac{1}{\theta^{\prime}}}\mathrm{d}\tau\\
&  \leq N^{\prime}(c_{\Omega})^{1-N^{\prime}}(c_{\Omega}s^{1/N^{\prime}%
})^{\frac{1}{\theta^{\prime}}-\frac{1}{(N-1)^{\prime}}}%
{\displaystyle\int_{0}^{c_{\Omega}s^{1/N^{\prime}}}}
k^{\ast\ast}(\tau)\tau^{-\frac{1}{\theta^{\prime}}}\mathrm{d}\tau\\
&  \leq(c_{\Omega})^{-1/\theta}N^{\prime}\left\Vert k\right\Vert _{\theta
,1}s^{\frac{\theta-(N-1)}{\theta N}}.
\end{align*}

Let $\overline{s}_{p}$ be defined by the equation%
\[
(c_{\Omega})^{-1/\theta}N^{\prime}\left\Vert k\right\Vert _{\theta
,1}(\overline{s}_{p})^{\frac{\theta-(N-1)}{\theta N}}=\frac{\xi_{p}}{K_{p}}%
\]
so that
\begin{equation}
\overline{s}_{p}:=\left(  \frac{(c_{\Omega})^{1/\theta}}{N^{\prime}\left\Vert
k\right\Vert _{\theta,1}}\frac{\xi_{p}}{K_{p}}\right)  ^{\frac{\theta
N}{\theta-(N-1)}}. \label{sbar}%
\end{equation}

Let us set%
\begin{equation}
C_{N,\Omega}:=\min\left\{  \left(  \frac{\left\vert \partial\Omega\right\vert
}{c_{\Omega}}\right)  ^{N^{\prime}}\;,\;\frac{\left\vert \Omega\right\vert
}{2}\;,\;C^{\prime\prime}\right\}  \label{CNO}%
\end{equation}
and then choose%

\begin{equation}
s_{p}:=\min\left\{  C_{N,\Omega}\text{ },\text{ }\overline{s}_{p}\right\}
\label{s0}%
\end{equation}
and
\[
t_{p}:=\left\vert \nabla u\right\vert ^{\ast}(s_{p}).
\]

Of course, $t_{p}\in\left[  \left\vert \nabla u\right\vert ^{\ast}(\left\vert
\Omega\right\vert /2),\left\Vert \nabla u\right\Vert _{\infty}\right]  .$
Moreover, as $\left\vert \nabla u\right\vert $ and $\left\vert \nabla
u\right\vert ^{\ast}$ are equidistributed we have that $\mu(t_{p})\leq s_{p}$
and this implies that%
\[
G(\mu(t_{p}))\leq G(s_{p})\leq G(\overline{s}_{p})\leq(c_{\Omega})^{-1/\theta
}N^{\prime}\left\Vert k\right\Vert _{\theta,1}(\overline{s}_{p})^{\frac
{\theta-(N-1)}{\theta N}}=\frac{\xi_{p}}{K_{p}}.
\]
Consequently,%
\[
K_{p}%
{\displaystyle\int_{0}^{\mu(t_{p})}}
k^{\ast\ast}(c_{\Omega}r^{\frac{1}{N^{\prime}}})r^{-\frac{1}{N^{\prime}}%
}\mathrm{d}r=K_{p}G(\mu(t_{p}))\leq\xi_{p}.
\]

This estimate in (\ref{aux3}), with $t_{0}=t_{p},$ yields%
\begin{equation}
F_{\epsilon}(\left\Vert \nabla u\right\Vert _{\infty})\leq CF_{\epsilon}%
(t_{p})+\dfrac{C}{\xi_{p}}b_{\epsilon}(\left\Vert \nabla u\right\Vert
_{\infty})\left\Vert \nabla u\right\Vert _{\infty}\left\Vert f\right\Vert
_{N,1}+\dfrac{C}{(\xi_{p})^{2}}\left\Vert \nabla u\right\Vert _{\infty
}\left\Vert f\right\Vert _{N,1}^{2}, \label{aux4}%
\end{equation}
for a positive constant $C$ sufficiently large, that depends at most on $N$
and $\Omega$ (but not on $p$).

Hence, noticing from (\ref{Fb}) that
\[
(K_{p})^{-1}\left\Vert \nabla u\right\Vert _{\infty}b_{\epsilon}(\left\Vert
\nabla u\right\Vert _{\infty})^{2}\leq F_{\epsilon}(\left\Vert \nabla
u\right\Vert _{\infty})
\]
and
\[
F_{\epsilon}(t_{p})\leq t_{p}b_{\epsilon}(t_{p})^{2}\leq\left\Vert \nabla
u\right\Vert _{\infty}b_{\epsilon}(t_{p})^{2}%
\]
we obtain from (\ref{aux4}) the inequality (after canceling $\left\Vert \nabla
u\right\Vert _{\infty}$)$.$
\begin{equation}
b_{\epsilon}(\left\Vert \nabla u\right\Vert _{\infty})^{2}\leq CK_{p}%
b_{\epsilon}(t_{p})^{2}+CK_{p}b_{\epsilon}(\left\Vert \nabla u\right\Vert
_{\infty})\frac{\left\Vert f\right\Vert _{N,1}}{\xi_{p}}+CK_{p}\left(
\frac{\left\Vert f\right\Vert _{N,1}}{\xi_{p}}\right)  ^{2} \label{aux5}%
\end{equation}
for a positive constant $C$ sufficiently large, that depends at most on $N$
and $\Omega$ (but not on $p$).

The following lemma is elementary.

\begin{lemma}
\label{square}If $X,$ $x,$ $Y$ and $c$ are positive numbers satisfying%
\[
X^{2}\leq cx^{2}+cYX+cY^{2},
\]
then
\[
X\leq\sqrt{c}x+(c+1)Y.
\]

\end{lemma}

\begin{proof}
One has%
\begin{align*}
\left(  X-\frac{cY}{2}\right)  ^{2}  &  \leq cx^{2}+cY^{2}+\left(  \frac
{cY}{2}\right)  ^{2}\\
&  \leq cx^{2}+(\frac{c}{2}+1)^{2}Y^{2}\leq\left(  \sqrt{c}x+(\frac{c}%
{2}+1)Y\right)  ^{2},
\end{align*}
so that%
\[
\left\vert X-\frac{cY}{2}\right\vert \leq\sqrt{c}x+(\frac{c}{2}+1)Y.
\]

Hence, if $X-\dfrac{cY}{2}\geq0,$ then%
\[
X\leq\sqrt{c}x+(\frac{c}{2}+1+\frac{c}{2})Y=\sqrt{c}x+(c+1)Y
\]
and if $X-\dfrac{cY}{2}<0,$ then%
\[
X<\dfrac{cY}{2}\leq\sqrt{c}x+(c+1)Y.
\]

\end{proof}

Lemma \ref{square} allows us to deduce from (\ref{qsi}) and (\ref{aux5}) that%
\begin{equation}
b_{\epsilon}(\left\Vert \nabla u\right\Vert _{\infty})\leq\sqrt{CK_{p}%
}b_{\epsilon}(t_{p})+\frac{2\left(  CK_{p}+1\right)  }{\min\left\{
p-1,1\right\}  }\left\Vert f\right\Vert _{N,1} \label{aux8}%
\end{equation}
where $K_{p}$ is defined by (\ref{C'p}), and $C$ is a positive constant
depending at most on $\Omega$ and $N$.

Now, let $\beta_{\epsilon}:[0,\infty)\rightarrow\lbrack0,\infty)$ be the
function defined by $\beta_{\epsilon}(t):=tb_{\epsilon}(t)$ and recall the
function $\psi_{\epsilon}$ defined in (\ref{psi}). It follows from (\ref{Bb})
and Lemma \ref{2.14} that%
\begin{align*}
\int_{\Omega}\beta_{\epsilon}(\left\vert \nabla u\right\vert )\mathrm{d}x  &
=\int_{\Omega}\left\vert \nabla u\right\vert b_{\epsilon}(\left\vert \nabla
u\right\vert )\mathrm{d}x\\
&  \leq\max\left\{  2,p\right\}  \int_{\Omega}B_{\epsilon}(\left\vert \nabla
u\right\vert )\mathrm{d}x\leq C^{\prime\prime}S_{1}\psi_{\epsilon}(\left\Vert
f\right\Vert _{N,1})
\end{align*}
where%
\begin{equation}
S_{1}:=\max\left\{  2,p\right\}  C_{p}=\left\{
\begin{array}
[c]{lll}%
2^{\frac{p}{p-1}} & \mathrm{if} & 1<p<2\\
p^{2} & \mathrm{if} & p\geq2.
\end{array}
\right.  \label{S1}%
\end{equation}
Hence, as
\[
\int_{\Omega}\beta_{\epsilon}(\left\vert \nabla u\right\vert )\mathrm{d}%
x\geq\int_{\left\{  \left\vert \nabla u\right\vert \geq t_{p}\right\}  }%
\beta_{\epsilon}(\left\vert \nabla u\right\vert )\mathrm{d}x\geq
\beta_{\epsilon}(t_{p})\lim_{t\rightarrow t_{p}^{-}}\mu(t)\geq\beta_{\epsilon
}(t_{p})s_{p}%
\]
we have that%
\[
\beta_{\epsilon}(t_{p})\leq S_{2}\psi_{\epsilon}(\left\Vert f\right\Vert
_{N,1})
\]
where%
\begin{equation}
S_{2}:=\frac{C^{\prime\prime}S_{1}}{s_{p}}. \label{S2}%
\end{equation}

Thus, after noticing from (\ref{s0}) that $\frac{C^{\prime\prime}}{s_{p}}%
\geq1$ and $S_{1}\geq1,$ we obtain from Lemma \ref{Lem1} the estimate
\begin{equation}
\beta_{\epsilon}(t_{p})\leq\psi_{\epsilon}(S_{2}\left\Vert f\right\Vert
_{N,1}). \label{aux6}%
\end{equation}

As
\[
b_{\epsilon}(\beta_{\epsilon}^{-1}(\psi_{\epsilon}(s))=s\text{ \ for }s\geq0
\]
it follows from (\ref{aux6}) that
\[
b_{\epsilon}(t_{p})\leq b_{\epsilon}(\beta_{\epsilon}^{-1}(\psi_{\epsilon
}(S_{2}\left\Vert f\right\Vert _{N,1}))=S_{2}\left\Vert f\right\Vert _{N,1}.
\]

Thus, (\ref{aux8}) yields%
\begin{equation}
b_{\epsilon}(\left\Vert \nabla u\right\Vert _{\infty})\leq S_{3}\left\Vert
f\right\Vert _{N,1} \label{aux9a}%
\end{equation}
where%
\begin{equation}
S_{3}:=\sqrt{CK_{p}}S_{2}+\frac{2\left(  CK_{p}+1\right)  }{\min\left\{
p-1,1\right\}  }. \label{S3}%
\end{equation}

\begin{remark}
We note from (\ref{C'p}), (\ref{qsi}) and (\ref{sbar}) that
\begin{equation}
\overline{s}_{p}=\overline{C}\left\{
\begin{array}
[c]{lll}%
\left(  \frac{p-1}{6}\right)  ^{\frac{\theta N}{\theta-(N-1)}} & \mathrm{if} &
1<p<2\\
\left(  \frac{1}{2(2p-1)}\right)  ^{\frac{\theta N}{\theta-(N-1)}} &
\mathrm{if} & p\geq2
\end{array}
\right.  \label{stil}%
\end{equation}
where%
\begin{equation}
\overline{C}:=\left(  \frac{(c_{\Omega})^{1/\theta}}{N^{\prime}\left\Vert
k\right\Vert _{\theta,1}}\right)  ^{\frac{\theta N}{\theta-(N-1)}}.
\label{Cbar}%
\end{equation}

\end{remark}

\begin{remark}
\label{spp}Note from (\ref{stil}) that
\[
\lim_{p\rightarrow1^{+}}\overline{s}_{p}=\lim_{p\rightarrow+\infty}%
\overline{s}_{p}=0.
\]
Thus, we can see from (\ref{s0}) that if $p$ is sufficiently close to $1$ or
sufficiently greater than $2,$ then $s_{p}=\overline{s}_{p}.$
\end{remark}

\begin{remark}
\label{rem2}According to Remark \ref{rem1}, if $\Omega$ is convex the latter
term on right-hand side of (\ref{aux2}) can be discarded and so we can take
$s_{p}=\left\vert \Omega\right\vert /2.$
\end{remark}

Now, let us again denote by $u_{\epsilon}$ either the weak solution to
(\ref{nae}), under the compatibility condition (\ref{f0}), or the weak
solution to (\ref{dae}). Thus, (\ref{aux9a}) can be written as%
\begin{equation}
b_{\epsilon}(\left\Vert \nabla u_{\epsilon}\right\Vert _{\infty})\leq
S_{3}\left\Vert f\right\Vert _{N,1}. \label{aux9}%
\end{equation}

As in Step 2 of \cite[Section 4]{CM} the assumption (\ref{regb}) can be
removed by properly approximating $\Omega$ by a sequence of smooth domains
(and convex if $\Omega$ is convex). Once removed (\ref{regb}) the weak
solution $u$ to either (\ref{nf}) or (\ref{df}) is obtained as limit of
$u_{\epsilon}$ as $\epsilon\rightarrow0.$ This can be done as in Step 3 of
\cite[Section 4]{CM} by applying known regularity results (see \cite{Li91})
after taking into account the convergences (\ref{bconv})-(\ref{aconv}).

Hence, by letting $\epsilon\rightarrow0^{+}$ in (\ref{aux9}) we arrive at the
estimate%
\begin{equation}
\left\Vert \nabla u\right\Vert _{\infty}^{p-1}\leq S_{3}\left\Vert
f\right\Vert _{N,1} \label{Cdp1}%
\end{equation}
with $S_{3}$ defined by (\ref{S3}) and $f$ fulfilling (\ref{regf}).

Then, proceeding as in Step 4 of \cite[Section 4]{CM} the assumption
(\ref{regf}) can be removed by density arguments.

Finally, let us estimate $S_{3}$ in terms of $p.$ From now on $C$ will denote
a positive constant sufficiently large possibly depend on $N$ and $\Omega,$
but not on $p.$ We consider the dependence on $\left\Vert k\right\Vert
_{\theta,1}$ and $\theta$ as part of the dependence on $\Omega.$

Let
\[
\overline{C}_{N,\Omega}:=\max\left\{  (C_{N,\Omega})^{-1},(\overline{C}%
)^{-1}\right\}
\]
where $C_{N,\Omega}$ and $\overline{C}$ are defined by (\ref{CNO}) and
(\ref{Cbar}), respectively.

If $1<p<2,$ then (\ref{s0}) implies that either%
\[
\frac{1}{s_{p}}=\frac{1}{C_{N,\Omega}}\leq\frac{1}{C_{N,\Omega}}\left(
\frac{p-1}{6}\right)  ^{-\frac{\theta N}{\theta-(N-1)}}\leq\overline
{C}_{N,\Omega}\left(  \frac{p-1}{6}\right)  ^{-\frac{\theta N}{\theta-(N-1)}}%
\]
or, according to (\ref{stil}),
\[
\frac{1}{s_{p}}=\frac{1}{\overline{s}_{p}}=\frac{1}{\overline{C}}\left(
\frac{p-1}{6}\right)  ^{-\frac{\theta N}{\theta-(N-1)}}\leq\overline
{C}_{N,\Omega}\left(  \frac{p-1}{6}\right)  ^{-\frac{\theta N}{\theta-(N-1)}%
}.
\]
Thus, (\ref{S1}) and (\ref{S2}) yield%
\[
S_{2}\leq C2^{\frac{p}{p-1}}(p-1)^{-\frac{\theta N}{\theta-(N-1)}},
\]
where $6^{(\theta N)/(\theta-(N-1))}\overline{C}_{N,\Omega}$ and
$C^{\prime\prime}$ are absorbed by $C$ (recall that $C^{\prime\prime}$ also
depends at most on $N$ and $\Omega$).

Then, as $K_{p}=3$ and
\[
(p-1)^{-1}<2^{\frac{p}{p-1}}(p-1)^{-\frac{\theta N}{\theta-(N-1)}}%
\]
(\ref{S3}) yields
\begin{equation}
S_{3}\leq C2^{\frac{p}{p-1}}(p-1)^{-\frac{\theta N}{(\theta-(N-1))}}
\label{S3a}%
\end{equation}
which combined with (\ref{Cdp1}) produces (\ref{Cdp}).

Similarly, if $p\geq2$ then (\ref{s0}) and (\ref{stil}) imply that%
\[
\frac{1}{s_{p}}\leq\overline{C}_{N,\Omega}(2(2p-1))^{\frac{\theta N}%
{\theta-(N-1)}}<4^{\frac{\theta N}{\theta-(N-1)}}\overline{C}_{N,\Omega
}p^{\frac{\theta N}{\theta-(N-1)}}.
\]
Hence, it follows from (\ref{S1}) and (\ref{S2}) that
\[
S_{2}\leq Cp^{2+\frac{\theta N}{\theta-(N-1)}}.
\]

Consequently,%
\begin{align*}
S_{3}  &  =\sqrt{C(2p-1)}S_{2}+\frac{2\left(  C(2p-1)+1\right)  }{\min\left\{
p-1,1\right\}  }\\
&  =\sqrt{C(2p-1)}S_{2}+2\left(  C(2p-1)+1\right)  \leq Cp^{\frac{1}{2}}S_{2},
\end{align*}
that is,%
\begin{equation}
S_{3}\leq Cp^{\left(  2+\frac{\theta N}{\theta-(N-1)}\right)  +\frac{1}{2}%
}=Cp^{\frac{5}{2}+\frac{\theta N}{\theta-(N-1)}}. \label{S3b}%
\end{equation}
Therefore, (\ref{Cdp}) follows from (\ref{Cdp1}) and (\ref{S3b}).

According to Remark \ref{rem2} in the case where $\Omega$ is convex
(\ref{CCdp}) follows from the fact that (\ref{S3a}) and (\ref{S3b}) can be
respectively replaced with $S_{3}\leq C2^{\frac{p}{p-1}}$ and $S_{3}\leq
Cp^{\frac{5}{2}}.$

\subsection{Proof of Theorem \ref{N=2}\label{sec3.2}}

In this subsection we consider $N=2,$ $\theta>1$ and $q>2.$ Again, we
initially assume the regularity assumptions (\ref{regb}) and (\ref{regf}).

The inequality (\ref{aux2}) now writes as
\begin{equation}%
\begin{array}
[c]{lll}%
2\xi_{p}F_{\epsilon}(\left\Vert \nabla u\right\Vert _{\infty}) & \leq &
2\xi_{p}F_{\epsilon}(t_{0})+C_{\Omega}b_{\epsilon}(\left\Vert \nabla
u\right\Vert _{\infty})\left\Vert \nabla u\right\Vert _{\infty}%
{\displaystyle\int_{0}^{\left\vert \Omega\right\vert }}
r^{-1/2}\phi(r)\mathrm{d}r\\
&  & +\dfrac{C_{\Omega}}{\xi_{p}}\left\Vert \nabla u\right\Vert _{\infty}%
{\displaystyle\int_{0}^{\left\vert \Omega\right\vert }}
r^{-1}%
{\displaystyle\int_{0}^{r}}
f^{\ast}(\rho)^{2}\mathrm{d}\rho\mathrm{d}r\\
&  & +K_{p}F_{\epsilon}(\left\Vert \nabla u\right\Vert _{\infty})%
{\displaystyle\int_{0}^{\mu(t_{0})}}
k^{\ast\ast}(c_{\Omega}r^{\frac{1}{2}})r^{-\frac{1}{2}}\mathrm{d}r
\end{array}
\label{aux2a}%
\end{equation}
whenever $t_{0}\in\left[  \left\vert \nabla u\right\vert ^{\ast}%
(\Omega/2),\left\Vert \nabla u\right\Vert _{\infty}\right]  ,$ where $k$ is a
nonnegative function in $L^{\theta,1}(\partial\Omega).$

Let us derive the following estimates that respectively correspond to
(\ref{fa}) and (\ref{fb}):
\begin{equation}
\int_{0}^{\left\vert \Omega\right\vert }r^{-1/2}\phi(r)\mathrm{d}r\leq
C_{q,\Omega}\left\Vert f\right\Vert _{q} \label{fa2}%
\end{equation}
and%
\begin{equation}
\int_{0}^{\left\vert \Omega\right\vert }r^{-1}%
{\displaystyle\int_{0}^{r}}
f^{\ast}(\rho)^{2}\mathrm{d}\rho\mathrm{d}r\leq\widetilde{C}_{q,\Omega
}\left\Vert f\right\Vert _{q}^{2} \label{fb2}%
\end{equation}
where $C_{q,\Omega}$ and $\widetilde{C}_{q,\Omega}$ are constants that depend
only on $q$ and $\left\vert \Omega\right\vert .$

First, we note from H\"{o}lder inequality that
\begin{equation}
\int_{0}^{r}f^{\ast}(\tau)^{2}d\tau\leq\left\Vert f\right\Vert _{q}%
^{2}r^{1-(2/q)},\text{ }r\geq0. \label{faa}%
\end{equation}

Proceeding as in \cite[Lemma 3.5]{CM}, we obtain the following inequalities
\begin{align*}
\int_{0}^{\left\vert \Omega\right\vert }r^{-1/2}\phi(r)\mathrm{d}r  &
\leq\int_{0}^{\left\vert \Omega\right\vert }r^{-1/2}\phi^{\ast}(r)\mathrm{d}%
r\\
&  =\int_{0}^{\left\vert \Omega\right\vert }r^{-1/2}\left(  \phi^{\ast}%
(r)^{2}\right)  ^{1/2}\mathrm{d}r\\
&  \leq\int_{0}^{\left\vert \Omega\right\vert }r^{-1/2}\left(  \frac{1}{r}%
\int_{0}^{r}\phi^{\ast}(\tau)^{2}\mathrm{d}\tau\right)  ^{1/2}\mathrm{d}r\\
&  =\int_{0}^{\left\vert \Omega\right\vert }r^{-1}\left(  \int_{0}^{r}%
\phi^{\ast}(\tau)^{2}\mathrm{d}\tau\right)  ^{1/2}\mathrm{d}r.
\end{align*}
Here, we have used the Hardy-Littlewood inequality and the fact that
$\phi^{\ast}$ is decreasing.

Hence, the estimate (\ref{fa2}) follows after using (\ref{phif}) and
(\ref{faa}):
\begin{align*}
\int_{0}^{\left\vert \Omega\right\vert }r^{-1/2}\phi(r)\mathrm{d}r  &
\leq\int_{0}^{\left\vert \Omega\right\vert }r^{-1}\left(  \int_{0}^{r}f^{\ast
}(\tau)^{2}\mathrm{d}\tau\right)  ^{1/2}\mathrm{d}r\\
&  \leq\left\Vert f\right\Vert _{q}\int_{0}^{\left\vert \Omega\right\vert
}r^{-(1/2)-(1/q)}\mathrm{d}r=C_{q,\Omega}\left\Vert f\right\Vert _{q}.
\end{align*}
where $C_{q,\Omega}:=\frac{\left\vert \Omega\right\vert ^{(1/2)-(1/q)}%
}{(1/2)-(1/q)}.$

Estimate (\ref{fb2}) also stems from (\ref{faa}), since%
\begin{align*}
\int_{0}^{\left\vert \Omega\right\vert }r^{-1}%
{\displaystyle\int_{0}^{r}}
f^{\ast}(\rho)^{2}\mathrm{d}\rho\mathrm{d}r  &  \leq\left\Vert f\right\Vert
_{q}^{2}\int_{0}^{\left\vert \Omega\right\vert }r^{-1}r^{1-(2/q)}%
\mathrm{d}\rho\mathrm{d}r\\
&  =\left\Vert f\right\Vert _{q}^{2}\frac{q}{q-2}\left\vert \Omega\right\vert
^{\frac{q-2}{q}}=\widetilde{C}_{q,\Omega}\left\Vert f\right\Vert _{q}^{2}%
\end{align*}
where $\widetilde{C}_{q,\Omega}:=\frac{q}{q-2}\left\vert \Omega\right\vert
^{\frac{q-2}{q}}.$

Now, using (\ref{fa}) and (\ref{fb}) in (\ref{aux2a}) we obtain%
\begin{equation}%
\begin{array}
[c]{lll}%
2\xi_{p}F_{\epsilon}(\left\Vert \nabla u\right\Vert _{\infty}) & \leq &
2\xi_{p}F_{\epsilon}(t_{0})+C_{\Omega}C_{q,\Omega}b_{\epsilon}(\left\Vert
\nabla u\right\Vert _{\infty})\left\Vert \nabla u\right\Vert _{\infty
}\left\Vert f\right\Vert _{q}\\
&  & +\dfrac{C_{\Omega}\widetilde{C}_{q,\Omega}}{\xi_{p}}\left\Vert \nabla
u\right\Vert _{\infty}\left\Vert f\right\Vert _{q}^{2}\\
&  & +K_{p}F_{\epsilon}(\left\Vert \nabla u\right\Vert _{\infty})%
{\displaystyle\int_{0}^{\mu(t_{0})}}
k^{\ast\ast}(c_{\Omega}r^{\frac{1}{2}})r^{-\frac{1}{2}}\mathrm{d}r
\end{array}
\label{aux3a}%
\end{equation}
whenever $t_{0}\in\left[  \left\vert \nabla u\right\vert ^{\ast}%
(\Omega/2),\left\Vert \nabla u\right\Vert _{\infty}\right]  .$

Let us define $G:[0,\infty)\rightarrow\lbrack0,\infty)$ as%
\[
G(s):=%
{\displaystyle\int_{0}^{s}}
k^{\ast\ast}(c_{\Omega}r^{\frac{1}{2}})r^{-\frac{1}{2}}\mathrm{d}r=\frac
{2}{c_{\Omega}}%
{\displaystyle\int_{0}^{c_{\Omega}s^{1/2}}}
k^{\ast\ast}(\tau)\mathrm{d}\tau,\text{ \ }s>0.
\]

If $c_{\Omega}s^{1/2}\leq\left\vert \partial\Omega\right\vert ,$ then%
\[
G(s)\leq\frac{2}{c_{\Omega}}(c_{\Omega}s^{1/2})^{\frac{1}{\theta^{\prime}}}%
{\displaystyle\int_{0}^{c_{\Omega}s^{1/2}}}
k^{\ast\ast}(\tau)\tau^{-\frac{1}{\theta^{\prime}}}\mathrm{d}\tau
\leq2(c_{\Omega})^{-\frac{1}{\theta}}s^{\frac{1}{2\theta^{\prime}}}\left\Vert
k\right\Vert _{\theta,1}.
\]
Hence, we define $\overline{s}_{p}$ by the equality%
\[
2(c_{\Omega})^{-\frac{1}{\theta}}(\overline{s}_{p})^{\frac{1}{2\theta^{\prime
}}}\left\Vert k\right\Vert _{\theta,1}=\frac{\xi_{p}}{K_{p}},
\]
so that
\begin{equation}
\overline{s}_{p}=\overline{C}\left\{
\begin{array}
[c]{lll}%
\left(  \frac{p-1}{6}\right)  ^{\frac{2\theta}{\theta-1}} & \mathrm{if} &
1<p<2\\
\left(  \frac{1}{2(2p-1)}\right)  ^{\frac{2\theta}{\theta-1}} & \mathrm{if} &
p\geq2
\end{array}
\right.  \label{spa}%
\end{equation}
where $\overline{C}:=(2\left\Vert k\right\Vert _{\theta,1}(c_{\Omega}%
)^{-\frac{1}{\theta}})^{-\frac{2\theta}{\theta-1}}.$

Let $t_{p}:=\left\vert \nabla u\right\vert ^{\ast}(s_{p})$ where
\begin{equation}
s_{p}:=\min\left\{  C_{\Omega}\text{ },\text{ }\overline{s}_{p}\right\}  .
\label{sp2}%
\end{equation}
and%
\[
C_{\Omega}:=\min\left\{  \left(  \frac{\left\vert \partial\Omega\right\vert
}{c_{\Omega}}\right)  ^{2}\;,\;\frac{\left\vert \Omega\right\vert }%
{2}\;,\;C^{\prime\prime}\right\}
\]

It follows that $\mu(t_{p})\leq s_{p}$ and
\begin{align*}
K_{p}%
{\displaystyle\int_{0}^{\mu(t_{p})}}
k^{\ast\ast}(c_{\Omega}r^{\frac{1}{2}})r^{-\frac{1}{2}}\mathrm{d}r  &
=K_{p}G(\mu(t_{p}))\\
&  \leq K_{p}G(s_{p})\leq K_{p}2(c_{\Omega})^{\frac{1}{\theta}}(\overline
{s}_{p})^{1/(2\theta^{\prime})}\left\Vert k\right\Vert _{\theta,1}=\xi_{p}.
\end{align*}
Taking $t_{0}=t_{p}$ in (\ref{aux3a}), using (\ref{Fb}), the latter estimate
and Lemma \ref{square} we arrive at%
\[
b_{\epsilon}(\left\Vert \nabla u\right\Vert _{\infty})\leq\sqrt{CK_{p}%
}b_{\epsilon}(t_{p})+\frac{2\left(  CK_{p}+1\right)  }{\min\left\{
p-1,1\right\}  }\left\Vert f\right\Vert _{q}%
\]
for a positive constant $C$ sufficiently large, that depends at most on $q$
and $\Omega$ (but not on $p$). Then, estimating $b_{\epsilon}(t_{p})$ as in
(\ref{aux6}) we derive the estimate $b_{\epsilon}(\left\Vert \nabla
u_{\epsilon}\right\Vert _{\infty})\leq S_{3}\left\Vert f\right\Vert _{q}$
where $S_{3}$ is defined by combining (\ref{S1}), (\ref{S2}), (\ref{S3}),
(\ref{spa}) and (\ref{sp2}).

Repeating the script of Subsection \ref{sec3.1} we can remove the additional
regularity\ assumptions (\ref{regb}) and (\ref{regf}) to obtain the estimate%
\begin{equation}
\left\Vert \nabla u\right\Vert _{\infty}^{p-1}\leq S_{3}\left\Vert
f\right\Vert _{q} \label{S32}%
\end{equation}
where $u$ is either a weak solution of the Neumann problem (\ref{nf}), under
the compatibility condition $\int_{\Omega}f(x)\mathrm{d}x=0,$ or a solution of
the Dirichlet problem (\ref{df}). In both cases, $f\in L^{q}(\Omega)$ with
$q>2.$ Then, as in the latter part of Subsection \ref{sec3.1} we obtain from
(\ref{S32}) the estimate (\ref{C2a}) in the case in which $\partial\Omega\in
W^{2}L^{\theta,1}$ and the estimate (\ref{C2b}) in the case which $\Omega$ is
convex (since $s_{p}=\left\vert \Omega\right\vert /2$).

\section{An explict estimate for more general operators\label{aeps1}}

Suppose that $a_{\epsilon}$ is an abstract regularization of $a$ (not
necessarily related to the $p$-Laplacian) satisfying the structural
conditions
\[
-1<i_{a_{\epsilon}}:=\inf_{t>0}\frac{ta_{\epsilon}^{\prime}(t)}{a_{\epsilon
}(t)}\leq s_{a_{\epsilon}}:=\sup_{t>0}\frac{ta_{\epsilon}^{\prime}%
(t)}{a_{\epsilon}(t)}<\infty.
\]

Inspecting the proofs given in Section \ref{sec2} and taking into account
Remark \ref{aeps} one can deduce that%
\[
b_{\epsilon}(\left\Vert \nabla u_{\epsilon}\right\Vert _{\infty})\leq
S_{3,_{\epsilon}}\left\Vert f\right\Vert _{N,1}%
\]
where%
\[
S_{3,_{\epsilon}}:=\sqrt{CK_{\epsilon}}S_{2,\epsilon}+\dfrac{2(CK_{\epsilon
}+1)}{1+\min\left\{  i_{a_{\epsilon}},0\right\}  },
\]
with
\[
S_{2,_{\epsilon}}\leq CS_{1,\epsilon}\left(  \dfrac{3+2s_{a_{\epsilon}}%
}{1+\min\left\{  i_{a_{\epsilon}},0\right\}  }\right)  ^{\frac{\theta
N}{\theta-(N-1)}}%
\]
and%
\[
S_{1,\epsilon}:=(2+s_{a_{\epsilon}})C_{\epsilon}=(2+s_{a_{\epsilon}}%
)^{\frac{2+i_{a_{\epsilon}}}{1+i_{a_{\epsilon}}}}>1.
\]

These facts imply that%
\[
b_{\epsilon}(\left\Vert \nabla u_{\epsilon}\right\Vert _{\infty})\leq
C\Lambda_{\epsilon}\left\Vert f\right\Vert _{N,1}%
\]
where%
\[
\Lambda_{\epsilon}:=\sqrt{(3+2s_{a_{\epsilon}})}(2+s_{a_{\epsilon}}%
)^{\frac{2+i_{a_{\epsilon}}}{1+i_{a_{\epsilon}}}}\left(  \dfrac
{3+2s_{a_{\epsilon}}}{1+\min\left\{  i_{a_{\epsilon}},0\right\}  }\right)
^{\frac{\theta N}{\theta-(N-1)}}+\dfrac{4+2s_{a_{\epsilon}}}{1+\min\left\{
i_{a_{\epsilon}},0\right\}  }%
\]
and $C$ is a constant depending at most on $N$ and $\Omega.$ If $\Omega$ is
convex, then the factor involving the power $\frac{\theta N}{\theta-(N-1)}$
must be replaced with $1.$ Moreover, if
\[
-1<\min\left\{  i_{a},0\right\}  \leq\min\left\{  i_{a_{\epsilon}},0\right\}
\leq i_{a_{\epsilon}}\leq s_{a_{\epsilon}}\leq\max\left\{  s_{a},0\right\}  ,
\]
then%
\[
\Lambda_{\epsilon}\leq\Lambda(i_{a},s_{a})
\]
where $\Lambda(i_{a},s_{a})$ is the explicit function of $i_{a}$ and $s_{a}$
obtained from $\Lambda_{\epsilon}$ by replacing $s_{a_{\epsilon}}$ with
$\max\left\{  s_{a},0\right\}  $ and both $i_{a_{\epsilon}}$ and $\min\left\{
i_{a_{\epsilon}},0\right\}  $ with $\min\left\{  i_{a},0\right\}  .$ For
example, if $\Omega$ is convex then%
\[
\Lambda(i_{a},s_{a})=\sqrt{(3+2\max\left\{  s_{a},0\right\}  )}(2+\max\left\{
s_{a},0\right\}  )^{1+\frac{1}{1+\min\left\{  i_{a},0\right\}  }}%
+\dfrac{4+2s_{a_{\epsilon}}}{1+\min\left\{  i_{a},0\right\}  }.
\]

Therefore, the results of Cianchi and Maz'ya lead to global estimate%
\[
b(\left\Vert \nabla u\right\Vert _{\infty})\leq C\Lambda(i_{a},s_{a}%
)\left\Vert f\right\Vert _{N,1}%
\]
for the solutions to (\ref{nf}) and (\ref{df}) under the hypotheses of
Theorems \ref{main} and \ref{main2}. A similar estimate, with $\left\Vert
f\right\Vert _{q}$ in the place of $\left\Vert f\right\Vert _{N,1},$ also
holds under the hypotheses of Theorem \ref{N=2}.

\section*{Acknowledgments}

The author thanks the support of Fapemig/Brazil (PPM-00137-18, RED-00133-21),
FAPDF/Brazil (04/2021) and CNPq/Brazil (305578/2020-0).

\end{document}